\documentclass[a4paper,12pt]{article}

\usepackage{amsthm,amssymb,amsbsy,amsmath,amsfonts,amssymb,amscd,mathrsfs}
\usepackage{t1enc}
\usepackage{inputenc,color, upgreek}

\newcommand{\abs}{\mathfrak{M}}
\newcommand{\nnu}{\upvarpi}

\theoremstyle{plain}
\newtheorem{theorem}{Theorem}[section]
\newtheorem{corollary}[theorem]{Corollary}
\newtheorem{lemma}[theorem]{Lemma}

\newtheorem{proposition}[theorem]{Proposition}
\theoremstyle{definition}
\newtheorem{definition}[theorem]{Definition}
\newtheorem{remark}[theorem]{Remark}

\def\vrexp{\stackrel{\longrightarrow}{\exp}}
\newcommand{\rexp}[3]{\,  \vrexp \int_0^{#1} #2 \, d #3}
\newcommand{\rexpi}[4]{\,  \vrexp \int_{#1}^{#2} #3 \, d #4}
\newcommand{\deinde}[1]{\frac{\partial}{\partial #1}}

\newcommand{\ad}{\operatorname{ad}}
\newcommand{\vep}{\varepsilon}
\newcommand{\eps}{\varepsilon}
\renewcommand{\H}{\mathcal{H}}

\newcommand{\pro}{\Upsilon} 

\newcommand{\spann}{\mathrm{span}}
\newcommand{\NN}{\mathbb{N}}
\newcommand{\CC}{{\mathbb{C}}}
\newcommand{\RR}{{\mathbb{R}}}

\renewcommand{\l}{\ell}

\newcommand{\crop}{\operatorname{Crop}}

\newcommand{\Lg}{$L^{1}$-bounded}
\newcommand{\CCC}{Lie--Galerkin Control Condition}
\newcommand{\CSC}{Lie--Galerkin Stalking Condition}

\begin{document}

\title{
Multi-input 
{S}chr\"odinger equation:  
controllability, tracking, and application  to the quantum 
angular momentum 
\footnote{This research has been supported by the European Research Council, ERC
StG 2009 ``GeCoMethods'', contract number 239748, by the ANR project
GCM, program ``Blanc'',
project number NT09-504490}}

\author{Ugo Boscain\footnote{
{\indent 
CNRS, CMAP, \'Ecole Polytechnique, Palaiseau,   France, \& Team GECO, INRIA Saclay, 
{\tt ugo.boscain@polytechnique.edu}
}}
 \qquad
 Marco Caponigro 
 \footnote{{\indent
Conservatoire National des Arts et M\'eti\'ers, \'Equipe M2N, Paris, France, {\tt \small marco.caponigro@cnam.fr}
 }}
 \qquad
                Mario Sigalotti\footnote{\indent
INRIA Saclay, Team GECO \& 
CMAP, \'Ecole Polytechnique, Palaiseau,   France,
 {\tt mario.sigalotti@inria.fr}}
}

\maketitle

\begin{abstract}
We present a sufficient condition for approximate controllability of the bilinear discrete-spectrum Schr\"odinger equation exploiting the use of several controls. The controllability result extends to simultaneous  controllability, approximate controllability in $H^{s}$, and tracking in modulus.
 The result is more general than those present in the literature even in the case of one control and permits to treat situations in which the spectrum of the uncontrolled operator is very degenerate (e.g. it has multiple eigenvalues or equal gaps among different pairs of eigenvalues).
We apply the general result 
to a rotating polar linear molecule, driven by three orthogonal external fields. 
A remarkable property of this model is the presence of infinitely many degeneracies and resonances in the spectrum  preventing the application of the results in the literature.  
\end{abstract}


\section{Introduction}

In this paper we study the controllability and the tracking problem for a multi-input bilinear Schr\"odinger equation
\begin{equation}
i\frac{d\psi}{dt}(t)=(H_0+u_1(t)H_1+\ldots +u_p(t)H_p)\psi(t)
\label{12}
\end{equation}
where $H_0, \ldots, H_{p}$ are self-adjoint operators on a Hilbert space $\cal H$ and  
the drift Schr\"odinger operator $H_0$ (the \emph{internal Hamiltonian})
has discrete spectrum. The control functions $u_1(\cdot),\ldots,u_p(\cdot)$, representing external fields, are real-valued  and $\psi(\cdot)$ takes values in the unit sphere of ${\cal H}$.

The controllability of system~(\ref{12}) is a well-established topic when the state space $\mathcal{H}$ is 
finite-dimensional (see for instance \cite{dalessandro-book} and reference therein), thanks to general controllability methods for left-invariant control systems on compact Lie groups (\cite{brock, jur,  JK81, GB82, AGK96}).

When $\mathcal{H}$ is infinite-dimensional,  it is known  that the bilinear Schr\"odinger equation is not exactly controllable (see \cite{BMS,turinici}). Hence, one has to look for weaker controllability properties as, for instance, approximate controllability or controllability between  eigenstates of the internal Hamiltonian $H_{0}$ (which are the most relevant physical states).
In certain cases, when  the space variable is one-dimensional, a description of reachable sets has been provided (see~\cite{beauchard-coron, camillo}).
In dimension larger than one  or for more general situations, the exact description of the reachable set appears to be more difficult and at the moment only approximate controllability results are available. 
Most of them are for the single-input case 
(see, in particular, \cite{noi, mirrahimi-aihp, Nersy, nersesyan,beauchard-nersesyan,  ancoranoi, fratelli-nersesyan}).
Such results are based on sufficient conditions for controllability that are generic~\cite{PriSig10, MasonSigalotti2010, nersesyan} even in the case $p=1$. Nevertheless, in many examples interesting for applications these conditions cannot be directly applied or controllability fails to hold, as a consequence of the symmetries of the system. Symmetries can induce degeneracies in the spectrum 
(e.g. multiple eigenvalues or presence of equal spectral gaps)  and reduce the coupling of eigenstates via the control. This happens, for instance, in a planar rotating molecule controlled by only one control~\cite[Section~8]{ancoranoi} which is not (approximately) controllable. 

The use of more than one control opens new controllability horizons.

Controllability results with more than one input have been obtained  for specific systems~\cite{ervedoza_puel,brockettetc} and some general approximate controllability results between eigenfunctions have been proved via adiabatic methods \cite{Boscain_Adami,adiabatiko}.
The first multi-input result via Lie-algebraic methods is given in~\cite[Section~8]{ancoranoi} where the problem of the  spectral degeneracies in the planar rotating molecule has been overcome associating with every $1$-dimensional slice of the set of admissible controls an invariant subspace of the state space $\H$ on which the single-input controllability result applies.
Anyhow, such a technique does not apply for more general cases. In the case of a rotating rigid symmetric $3$D molecule this application of this method is obstructed by the fact that 
eigenspaces may have arbitrarily large dimension, since at every higher energy level new degeneracies appear. In some sense  the strategy of~\cite[Section~8]{ancoranoi} does not fully exploit the potentialities of the geometric method based on the controllability of the Galerkin approximations.

In this paper,  we present a sufficient condition for controllability of the discrete-spectrum bilinear Schr\"odinger equation which can be applied to cases in which the spectrum of the internal Hamiltonian $H_{0}$ is very degenerate. The results fully exploit the use of more than one control and extend to simultaneous controllability, approximate controllability in $H^s$, and stalking. 
Proving that a system is a stalker (i.e. it permits to track in modulus any given trajectory; for precise definitions see Section~\ref{sec:mainresults}) is a crucial issue when describing systems  containing  dissipative levels (and the dissipation is not taken into account in the mathematical model). In this case, a strategy is to keep the population of the dissipative levels as low as possible during the transitions in order to minimize the effects of the dissipation (see for instance the STIRAP model~\cite{CH90, VHSB01, Gauthier_Jauslin}). 

The result presented in this paper is more general than those present in the literature even in the single-input case.  For instance it applies to the Laplace--Dirichlet operator on a compact interval of $\RR$ 
with a control term  of the type $u(t)x$. 
Let us mention that in~\cite{ancoranoi}, approximate simultaneous controllability of this model has been proved breaking the degeneracy between gaps among eigenstates through perturbation techniques.
Here we prove the approximate simultaneous controllability and stalking without perturbation arguments. 
In this framework, proving a controllability result without perturbation arguments 
allows to translate directly the constructive proof of the main result in an algorithm that provides explicit expressions of controls~\cite{Qtrack}.

\subsection{Brief description of the general results}

The main result of the paper is a sufficient condition for approximate simultaneous controllability 
which we call the \CCC\ (see Definition~\ref{def:CCC}).

Roughly speaking, both the sufficient condition proposed in~\cite{ancoranoi} and the one presented here are based on the idea of driving the system with control laws in resonance with spectral gaps of the internal Hamiltonian $H_{0}$. However, whereas in~\cite{ancoranoi} the only actions on the system obtained by resonance which are exploited for the controllability are those corresponding to elementary transitions between two eigenstates, no such a restriction is imposed in the \CCC\ (see Section~\ref{sec:buca}).

The \CCC\ ensures strong controllability properties for the Galerkin approximations. Indeed it provides controllability for a fixed Galerkin approximation while avoiding the transfer of population to higher energy levels. This allows estimates on the difference between the dynamics of the finite-dimensional Galerkin approximation and the ones of the original infinite-dimensional system. The \CCC\  also ensures a  bound on the $L^{1}$ norm of the control achieving controllability which is uniform with respect to the prescribed tolerance, when the required transfer is between finite combinations of eigenstates.

Under the \CSC, a slight modification of the \CCC, we can prove that any trajectory can be tracked in modulus (see Theorem~\ref{metalemma-tracking}). 

The \CCC\ under the additional assumption that the system is $s$-weakly coupled (see Definition~\ref{def:wc}) as introduced in~\cite{weakly}, 
allows to conclude that the system is controllable in $H^{s/2}$ (see Theorem~\ref{thm:controlHs}).

\subsection{Application to the quantum angular momentum}

Rotational molecular dynamics is one of the most important examples of quantum systems with an infinite-dimensional Hilbert space and a discrete spectrum. 
%
Molecular orientation and alignment are well-established topics in the quantum control of molecular dynamics both from the experimental and the theoretical point of view (see \cite{stapelfeldt,sugny,seideman} and references therein). For linear molecules driven by linearly polarized laser fields in gas phase, alignment means an increased probability direction along the polarization axis whereas orientation requires in addition the same (or opposite) direction as the polarization vector. 
A large amount of numerical simulations have been done in this domain but the mathematical part is not yet fully understood. From this perspective, the controllability problem is 
a necessary step towards comprehension.

 We focus in this paper on the control by external fields of the rotation of a rigid linear molecule in $\RR^3$. This control problem corresponds to the control of the Schr\"odinger equation on the unit sphere $S^2$. 
 We show that the system driven by three fields along the three axes is  
 approximately  controllable
 for arbitrarily small controls.

Up to normalization of physical quantities (in particular, in units such that $\hbar=1$), the dynamics are governed by the equation
\begin{align}
i\frac{\partial\psi(\theta,\varphi,t)}{\partial t}= -\Delta \psi(\theta,\varphi,t) + (u_1(t) \sin \theta\cos \varphi
+ u_2(t) \sin \theta\sin \varphi+ u_3(t) \cos \theta)\psi(\theta,\varphi,t)
\label{1}
\end{align}
where $\theta,\varphi$ are the spherical coordinates, which are related to the Euclidean coordinates through the identities 
$$
x = \sin \theta \cos \varphi, \quad y = \sin \theta\sin \varphi,\quad z = \cos \theta,
$$
while 
$\Delta$ is the Laplace--Beltrami operator on the sphere $S^2$ (called in this context the \emph{angular momentum operator}), 
i.e., 
$$
\Delta = \frac{1}{\sin \theta} \deinde{\theta} \left( \sin \theta \deinde{\theta}\right) + \frac{1}{\sin^2 \theta} \frac{\partial^2}{\partial \varphi^2}.
$$

The wavefunction $\psi$ evolves in the unit sphere ${\cal S}$ of ${\cal H}=L^2(S^2,\CC)$.

As a consequence of the general multi-input result presented in Section~\ref{sec:mainresults} we have that~\eqref{1} is approximately controllable with arbitrarily small controls. A stronger statement, including simultaneous controllability in $H^{s}$ and stalking, 
is given in Section~\ref{rotational}.
\begin{theorem}
\label{controllability}
For every  $\psi^0$, $\psi^1$  belonging to ${\cal S}$  and every $\delta>0$,
 there exist $T>0$ and $u\in L^\infty([0,T],[0,\delta]^{3})$, such that
 the solution $\psi(\cdot)$ of equation \eqref{1}, corresponding to the control $u$ and with initial condition 
 $\psi_0$, 
satisfies
 $\|\psi^1-\psi(T)\|<\eps$.
\end{theorem}

There are two main difficulties preventing the application to this system results previously  in the literature.
Firstly, we deal here with several control parameters, while those general results were specifically conceived for the single-input case. Notice that, because of symmetry obstructions, equation~\eqref{1} is not controllable with only two of the three controls $u_1$, $u_2$, $u_3$.  
Secondly, the general theory developed in \cite{noi, nersesyan,ancoranoi} is based on 
nonresonance conditions on the spectrum of the internal Hamiltonian. 
The Laplace--Beltrami operator on $S^2$, however, has a severely degenerate spectrum, since the $\ell$-th eigenvalue $-i\ell(\ell+1)$ has multiplicity $2\ell+1$.
In \cite{noi} we proposed a perturbation technique in order to overcome resonance relations in the spectrum of the drift. This technique was applied in \cite{noicdc} to the case of the orientation of a molecule confined in a plane driven by one control.  The planar case is already technically 
challenging and a generalization of the same technique to the case of three controls in 3D seems  hard to achieve.

\subsection{Structure of the paper}
The structure of the paper
is the following: in the next section 
we present the general multi-input abstract framework and the main abstract results.
In Section~\ref{rotational} we  apply them to system \eqref{1}. 
The proofs of the abstract results are contained in Sections~\ref{sec:prova1}, \ref{sec:tracking}, and \ref{6}.

\section{Framework and main results}\label{sec:mainresults}

Let $p\in \NN$, $\delta>0$, and $U=U_1\times \cdots \times U_p$ with either $U_j=[0,\delta]$ or $U_j=[-\delta,\delta]$.

\begin{definition}\label{def:system}
Let $\cal H$ be an infinite-dimensional Hilbert space with scalar product $\langle \cdot ,\cdot \rangle$ and  $A,B_{1}, \ldots,B_{p}$ be (possibly unbounded) skew-adjoint operators on $\H$,
with domains $D(A),$ $D(B_{1}), \ldots, D(B_{p})$. 
Let us introduce the 
controlled equation
\begin{equation} \label{eq:main}
\frac{d\psi}{dt}(t)=(A+u_1(t)B_1+ \cdots + u_p(t) B_p) \psi(t),  \quad u(t) \in U. 
\end{equation}

We say that $A$ satisfies  ($\mathbb{A}1$) if the following assumption is true:
\begin{description}
\item[($\mathbb{A}1$)] $A$ has discrete spectrum with infinitely many distinct eigenvalues (possibly degenerate).
\end{description}

Denote by $ \Phi$ a Hilbert basis $(\phi_k)_{k \in \NN}$  of $\cal H$ 
made of eigenvectors of $A$ associated with the family of eigenvalues $(i \lambda_{k})_{k \in \mathbb{N}}$ and  let $\mathcal{L}$ be the set of finite linear combinations of eigenstates, that is,
$$
\mathcal{L} = \bigcup_{k \in \NN}  \mathrm{span}\{\phi_{1},\ldots,\phi_{k}\}.
$$

We say that $(A,B_{1}, \ldots, B_{p},U,\Phi)$ satisfies $(\mathbb{A})$ 
if $A$ satisfies ($\mathbb{A}1$) and the following assumptions hold:
\begin{description}
\item[($\mathbb{A}2$)] $\phi_k \in D(B_{j})$ for every $k \in
\mathbb{N}, j=1,\ldots,p $;
\item[($\mathbb{A}3$)] 
$A+u_1B_1+ \cdots + u_p B_p:\mathcal{L} \to \mathcal{H}$ is essentially skew-adjoint for every $u\in U$.
\end{description}
\end{definition}

If $(A,B_{1}, \ldots, B_{p},U,\Phi)$ satisfies 
$(\mathbb{A})$ 
then, for every $(u_{1},\ldots,u_{p}) \in U$, $A+u_{1}B_{1} + \cdots + u_{p} B_{p}$ generates a subgroup $e^{t(A+u_{1}B_{1} + \cdots + u_{p} B_{p})}$ of the 
group of unitary operators $\mathbf{U}(\H)$. 
It is therefore possible to define the propagator $\pro^{u}_{T}$ at time $T$ of system~\eqref{1} associated with a $p$-uple of piecewise constant controls $u(\cdot)=(u_{1}(\cdot), \ldots, u_{p}(\cdot))$ by composition of flows of the type $e^{t(A+u_{1}B_{1} + \cdots + u_{p} B_{p})}$. If, moreover, $B_{1},\ldots, B_{p}$ are bounded operators then the definition can be extended by continuity to every $L^{\infty}$ control law (see~\cite[Theorem 2.5]{BMS}).

\begin{definition}\label{def:controllability}
Let $(A,B_{1}, \ldots, B_{p},U,\Phi)$ satisfy 
$(\mathbb{A})$.
We say that \eqref{eq:main} is \emph{approximately controllable}
if for every $\psi_0,\psi_1$ in the unit sphere of $\cal H$
 and every $\vep>0$ there exists a piecewise constant control function $u:[0,T] \to U$ such that 
$
\|\psi_1-  \pro^{u}_{T}(\psi_{0})\| <\vep.
$
\end{definition}

\begin{definition} \label{DEF_simultaneous_controllability}
Let $(A,B_{1}, \ldots, B_{p},U,\Phi)$ satisfy $(\mathbb{A})$. 
We say that \eqref{eq:main} is \emph{approximately simultaneously controllable} if for every $r$ in $ \NN$, $\psi_1,\ldots,\psi_{r}$ in $\H$, $\hat\Upsilon$ in $\mathbf{U}(\H)$, and $\vep>0$  there exists a piecewise constant control $u:[0,T]\rightarrow U$
such that
$$
\left \| \hat\Upsilon \psi_k - \pro^{u}_T \psi_k \right \|<\vep,\qquad k=1,\dots,r.
$$
If, moreover, 
for every $\psi_1,\ldots,\psi_{r} \in \mathcal{L}$ and $\hat\Upsilon \in \mathbf{U}(\H)$ such that 
$\hat\Upsilon \psi_{1}, \ldots, \hat\Upsilon \psi_{r}  \in \mathcal{L}$, 
there exists $K>0$ (not depending on $\vep$)
such that $u$ can be chosen to satisfy, in addition,  $\|u\|_{L^{1}} \leq K$, we say that 
 \eqref{eq:main} is \emph{\Lg\ approximately simultaneously controllable}. 
\end{definition}

This last definition of controllability  with  \emph{a priori} bound on the $L^{1}$-norm of the control achieving controllability has been observed in preceding works~\cite{ancoranoi,periodic}. It implies a stronger controllability property as shown in Section~\ref{sec:Hs}.

Due to presence of the internal Hamiltonian and the boundedness of the controls, it is not possible  in general to track, with arbitrarily precision, an unfeasible curve in $\mathcal{S}$. We introduce, then, the notion of \emph{stalker}, that is a system for which it is possible to track any given curve up to phases (both for a single initial condition and  in the spirit of simultaneous control).
This definition makes sense from the physical point of view, since tracking up to phases means imposing the population of all energy levels of $H_{0}$ along the evolution.

The identification up to phases of elements of $\H$ in the basis $\Phi = (\phi_{k})_{k\in \NN}$ can be introduced through the projection 
 $$
 \abs:\psi\mapsto \sum_{k\in \NN}|\langle \phi_k,\psi\rangle|\phi_k.
 $$

\begin{definition} \label{def:trackable}
Let $(A,B_{1}, \ldots, B_{p},U,\Phi)$ satisfy $(\mathbb{A})$.
We say that~\eqref{eq:main} is a \emph{stalker} if for every $r$ in $ \NN$, $\psi_1,\ldots,\psi_{r}$ in $\H$, $\hat\Upsilon:[0,T] \to \mathbf{U}(\H)$ continuous, with $\hat\Upsilon_{0} = \mathrm{Id}_{\H}$, and $\vep>0$  there exist 
an invertible continuous function
$\tau:[0,T] \to [0,T_{\tau}]$ 
and
a
piecewise constant control $u:[0,T_{\tau}]\rightarrow U$
such that
$$
\left \| \abs(\hat\Upsilon_{t} \psi_k) - \abs (\pro^{u}_{\tau(t)} \psi_k) \right \|<\vep,\qquad k=1,\dots,r,
$$
for every $t \in [0,T_\tau]$.
\end{definition}

\subsection{Notation}

For every $n$ in $\NN$, define the orthogonal projection
$$
 \pi_n: \H \ni \psi\mapsto \sum_{k=1}^{n} \langle \phi_k,\psi\rangle
\phi_k \in \H.
$$
Given a linear operator $Q$ on $\H$ we identify the linear operator
$\pi_{n} Q \pi_{n}$ preserving 
$\spann\{\phi_{1},\ldots, \phi_{n}\}$ with 
its  
$n \times n$ complex matrix representation with 
respect to the basis $(\phi_{1},\ldots, \phi_{n})$.
We define
$$
A^{(n)} = \pi_{n} A \pi_{n} \quad \mbox{ and } \quad B_j^{(n)} = \pi_{n} B_{j} \pi_{n},
$$
for every $j=1,\ldots,p$.

Let us introduce the set $\Sigma_{n}$ 
of spectral gaps associated with the $n$-dimensional Galerkin approximation as
$$
\Sigma_n = \{|\lambda_l - \lambda_k| \mid l,k = 1, \ldots, n\}.
$$

For every $\sigma \geq 0$, every $m\in\NN$, and every $m\times m$ matrix $M$, let
$$
{\cal E}_\sigma(M)=  (M_{l,k} \delta_{\sigma, |\lambda_l - \lambda_k|})_{l,k=1}^m.
$$
The $n\times n$ matrix ${\cal E}_\sigma(B^{(n)}_j)$, $j=1,\dots,p$, corresponds to the  ``activation'' of the spectral gap $\sigma$: it reflects the action of the convexification procedure detailed in the following sections, which annihilates all the matrix elements $(B^{(n)}_j)_{l,k}$ such that 
$|\lambda_l-\lambda_k|\ne \sigma$.

Define
\begin{equation}\label{sigmabar}
\Xi_n=\left\{ (\sigma,j) \in \Sigma_n \times \{1,\dots,p\}\mid \exists M\in \mathfrak{u}(n)  \mbox{ s.t. }\mathcal{E}_\sigma(B_j^{(N)}) = 
\left(
\begin{array}{c|c}
M&0\\ \hline 0 & *
\end{array}
\right)
\mbox{ for every } N>n
\right\}.
\end{equation}
The matrices $\mathcal{E}_\sigma(B_j^{(n)})$ for $(\sigma,j)\in\Xi_n$ 
correspond to 
``compatible dynamics'' for the $n$-dimensional Galerkin approximation (\emph{compatible}, that is, with higher dimensional Galerkin approximations).

\subsection{Controllability results}

Let
\begin{align*}
 \mathcal{V}_{n}^0 &= 
\left\{A^{(n)} 
\right\}
\cup \left\{{\cal E}_\sigma(B^{(n)}_j)\mid
(\sigma,j)\in \Xi_n 
\mbox{ and $j$ is such that }
(0,j)\in\Xi_n
\right\}\\
& \qquad 
\cup
\left\{ 
{\cal E}_\sigma(B^{(n)}_j)\mid
(\sigma,j)\in \Xi_n, \sigma \neq 0, U_{j} = [-\delta, \delta]
 \right\}.
\end{align*}

\begin{definition}\label{def:CCC}
Let $(A,B_{1}, \ldots, B_{p},U,\Phi)$ satisfy $(\mathbb{A})$. We say that the \emph{\CCC}\ holds if 
for every $n_0\in \mathbb{N}$ there exists $n> n_0$ 
such that
\begin{equation}\label{hypothesis}
\mathrm{Lie} \mathcal{V}_{n}^0 \supseteq \mathfrak{su}(n).
\end{equation}
\end{definition}

\begin{theorem}[Abstract multi-input controllability result]\label{metalemma}
Assume that $(\mathbb{A})$ holds true.
If the \CCC\ holds
then the  system
$$
\dot{x} = (A+ u_1B_1+ \cdots + u_p B_p)x,\quad u\in U,
$$
is \Lg\ approximately simultaneously controllable.
\end{theorem}

\subsection{Stalking results}

For every $\xi \in S^{1} \subset \CC$, consider the matrix operator $J_{\xi}$ such that
$$
\left(J_{\xi}(M)\right)_{j,k} = 
\begin{cases}
\xi M_{j,k} & \mbox{ if } \lambda_j<\lambda_k\\
0  & \mbox{ if }  \lambda_j=\lambda_k\\
\bar\xi M_{j,k} & \mbox{ if } \lambda_j>\lambda_k.
\end{cases}
$$

Let
$$
 \mathcal{V}_{n} = 
\left\{J_{\xi}({\cal E}_\sigma(B^{(n)}_j))\mid
(\sigma,j)\in \Xi_n, \sigma \neq 0, \xi\in S^{1}
\right\}.
$$

Notice that $\mathcal{V}_{n} \subset \mathfrak{su}(n)$.

\begin{definition}\label{def:CSC}
Let $(A,B_{1}, \ldots, B_{p},U,\Phi)$ satisfy $(\mathbb{A})$. 
We say that the \emph{\CSC}\ holds if 
for every $n_0\in \mathbb{N}$ there exists $n> n_0$ 
such that
\begin{equation}\label{hypothesis-t}
\mathrm{Lie} \mathcal{V}_{n} = \mathfrak{su}(n).
\end{equation}
\end{definition}

\begin{theorem}[Abstract multi-input tracking result]\label{metalemma-tracking}
Let $U_j=[-\delta,\delta]$ for some $\delta>0$ and every $j=1,\ldots,p$.
Assume 
that $(\mathbb{A})$ holds true. 
If the \CSC\ holds
then the  system
$$
\dot{x} = (A+ u_1B_1+ \cdots + u_p B_p)x,\quad u\in U,
$$
is a stalker.
\end{theorem}

\begin{remark}
If $U_{j} = [-\delta,\delta]$ for every $j=1,\ldots,p$,
then
the \CSC\ implies the \CCC, as it follows from the relation 
$$
\left[ A^{(n)},\mathcal{E}_{\sigma}(B_{j}^{(n)}) \right] = \sigma J_{i}(\mathcal{E}_{\sigma}(B_{j}^{(n)})).
$$
\end{remark}

\subsection{Controllability in higher norms}\label{sec:Hs}

We define for $s > 0$,
$$
|A|^{s} \psi = \sum_{n\in \NN} |\lambda_{n}|^{s} \langle \psi, \phi_{n}\rangle \phi_{n}
$$
for every $\psi$ belonging to
$$
D(|A|^{s}) = \left\{ \psi \in \H \mid \sum_{n\in \NN} |\lambda_{n}|^{2s} |\langle \psi, \phi_{n}\rangle|^{2} < +\infty  \right\}.
$$
For every $\psi \in D(|A|^{s})$ we can define  the $|A|^{s}$-norm (or simply $s$-norm) of $\psi$ by $\|\psi\|_{{s}} = \||A|^{s} \psi\|$. If $A$ is the Laplace--Dirichlet operator on some bounded domain of $\RR^{n}$ then the $s$-norm is equivalent to the $H^{2s}$-norm  on $D(|A|^{s})$.

\begin{definition}
Let $(A,B_1,\ldots,B_p, U, \Phi)$ satisfy Assumption~$(\mathbb{A})$ and let $s>0$. 
System~\eqref{eq:main} is
\emph{approximately simultaneously controllable} (respectively \emph{approximately controllable}) \emph{for the $s$-norm} if for every $\varepsilon>0$, $r \in \NN$ (respectively $r=1$),
$\psi_1,\ldots,\psi_r$ in $D(|A|^{s})$, and $\hat{\Upsilon} \in \mathbf{U}(\H)$  such that 
$\hat\Upsilon \psi_{1}, \ldots, \hat\Upsilon \psi_{r} \in D(|A|^{s})$ 
there exists a piecewise constant function
$u_\varepsilon: [0,T_\varepsilon]\to \RR$ such that
$$
\|\hat{\Upsilon}\psi_j-\Upsilon_{T_\varepsilon}^{u_\varepsilon}
\psi_j\|_s<\varepsilon,
$$
for every $j=1,\ldots,r$. 
\end{definition}

We say that $(A,B_{1}, \ldots, B_{p}, U, \Phi)$ satisfies $(\mathbb{A}')$
if it satisfies ($\mathbb{A}$) and the following additional assumptions hold:
\begin{itemize}
\item[$(\mathbb{A}4)$] the operator $i(A+u_{1}B_{1}+\cdots + u_{p}B_{p} )$ is bounded from below for every $u\in \RR^{p}$;
\item[$(\mathbb{A}5)$]  the sequence $(\lambda_{k})_{k \in \NN}$ is non-increasing and unbounded.
\end{itemize}

\begin{definition}\label{def:wc}
Let $(A,B_1,\ldots,B_p, U, \Phi)$  satisfy Assumption~$(\mathbb{A}')$ and
let $s>0$. Then $(A,B_1,\ldots,B_p)$ is
\emph{$s$-weakly-coupled}
if 
$D(|A+u_{1}B_{1}+\cdots+ u_{p}B_{p}|^{s/2}) = D(|A|^{s/2})$ for every $u \in \RR^{p}$ and  there exists $C$ such that
$$
|\Re\langle |A|^{s} \psi,B_{l} \psi \rangle| \leq C|\langle |A|^{s}\psi, \psi\rangle|
$$
for every $l=1,\ldots,p$, $\psi \in D(|A|^{s})$.
\end{definition}

The following result is a consequence of~\cite[Proposition~2]{weakly} and can be obtained by adapting the arguments proposed in \cite[Proposition~5]{weakly}. We provide its proof in Section~\ref{6}.

\begin{theorem}\label{thm:controlHs}
Let $(A,B_1,\ldots,B_p, U, \Phi)$  satisfy Assumption~$(\mathbb{A}')$ and $(A,B_1,\ldots,B_p)$ be $s$-weakly coupled for some $s>0$.
 If~\eqref{eq:main} is \Lg\ approximately simultaneously  controllable 
 then it is approximately simultaneously  controllable
for the $s/2$-norm.
\end{theorem}

As a direct consequence we have the following result generalizing~\cite[Proposition~5]{weakly}.

\begin{corollary}\label{cor:qwerty}
Let $(A,B_1,\ldots,B_p, U, \Phi)$  satisfy Assumption~$(\mathbb{A}')$ and $(A,B_1,\ldots,B_p)$ be $s$-weakly coupled for some $s>0$.
If the \CCC\ holds
then  system~\eqref{eq:main}
is approximately simultaneously  controllable
for the $s/2$-norm.
\end{corollary}

\subsection{Example: the infinite potential well}\label{sec:buca}

We present the case of a particle confined in the interval $(-1/2,1/2)$ as a toy model to compare  the result in~\cite{ancoranoi} and Theorem~\ref{metalemma-tracking} on a single-input system.
The model has been extensively studied by several authors in the last decade and it has been the first quantum system for which a positive controllability result has been obtained (see~\cite{beauchard-coron}).
 In~\cite{ancoranoi} an approximate  simultaneous controllability has been obtained with geometric methods and using perturbations techniques. Indeed this model presents several resonances  
preventing the direct application of the results in~\cite{ancoranoi}. 

The Schr\"{o}dinger equation writes
\begin{equation}\label{EQ_potential_well}
 i \frac{\partial \psi}{\partial t}=-\frac{1}{2} \frac{\partial^2 \psi}{\partial x^2} - u(t) x \psi(x,t)
\end{equation}
with the boundary conditions $\psi(-1/2,t)=\psi(1/2,t)=0$ for every $t \in \RR$. 
We consider controls $u(\cdot)$ piecewise constant with values in $U=[-\delta,\delta]$ for some $\delta>0$.

In this case $\H=L^2 \left ( (-1/2,1/2), \CC \right )$ endowed with the Hermitian product $\langle \psi_1,\psi_2\rangle= \int_{-1/2}^{1/2} \overline{\psi_1(x)} \psi_2(x) dx$. The operators $A$ and $B=B_{1}$ are defined by $A\psi= i\frac{1}{2} \frac{\partial^2 \psi}{\partial x^2}$ for every $\psi$ in 
$D(A)= (H_2 \cap  H_0^1 )\left((-1/2,1/2), \CC\right)$ and $B\psi=i x \psi$.        
A complete set of eigenfunctions of $A$ is 
$$ 
\phi_k(x)=
\begin{cases}
 \sqrt{2} \cos(k\pi x)& \mbox{ when } k \mbox{ is odd,}\\
 \sqrt{2} \sin(k\pi x)& \mbox{ when } k \mbox{ is even,}
\end{cases}
\qquad
k \in \NN,
$$
associated with the eigenvalues $i\lambda_k=  - i\frac{k^2 \pi^2}{2}, k \in \NN$.
 Notice that
$$
\langle \phi_{j}, B \phi_{k} \rangle \neq 0
$$
if and only if $j+k$ is odd. In particular $\mathcal{V}_{n} \subset\mathfrak{su}(n) $ for every $n$.

We prove by induction on $n$ that 
$
\mathrm{Lie} \mathcal{V}_{n} = \mathfrak{su}(n)$, and hence that the \CSC\ is fulfilled.
Notice that the matrices $\mathcal{E}_{2k-1}(B^{(N)})$, for $k \leq N$, have only zero elements in the positions $(j,l)$ for $j\leq N$ and $l\geq N+1$, since
$$
 l^{2}-j^{2} \geq (N+1)^{2} - N^{2} = 2N+1 > 2k-1.
$$
Hence $\mathcal{E}_{2k-1}(B^{(N)}) \in \mathcal{V}_{N}$ for $k=1,\ldots,N$. We prove the claim by showing that 
\begin{equation}\label{disamina}
\mathrm{Lie}\left( \{\mathcal{E}_{2k-1}(B^{(n)})\mid k=1,\ldots,n\} \right) = \mathfrak{su}(n).
\end{equation}
For $n=2$,
$\mathcal{E}_{3}(B^{(2)}) = 
 \begin{pmatrix}
 0 & b_{12}\\
 -\bar b_{12} & 0
 \end{pmatrix}
$
generates $\mathfrak{su}(2)$ because $b_{12}\neq0$.

Now assume that
$$ 
\mathrm{Lie}\left( \{\mathcal{E}_{2k-1}(B^{(n-1)})\mid k=1,\ldots,n-1\} \right) = \mathfrak{su}(n-1),
$$
and let us prove~\eqref{disamina}.
The matrices  $\mathcal{E}_{2k-1}(B^{(n)})$ are in $ \mathfrak{su}(n)$  for every $k=1,\ldots,n-1$ and generate the subalgebra of matrices in $ \mathfrak{su}(n)$  with zero elements in the $n$th row and $n$th column.
In particular there exists $M \in \mathrm{Lie}\left( \{\mathcal{E}_{2k-1}(B^{(n)})\,:\,k=1,\ldots,n-1\} \right) $
such that $M+\mathcal{E}_{2n-1}(B^{(n)}) $ has only two  nonzero elements in the positions $(n-1,n)$ and $(n,n-1)$. So
\begin{align*}
&\mathrm{Lie} (  \{\mathcal{E}_{2k-1}(B^{(n)})\mid k=1,\ldots,n\} )  \\
&\qquad \supset \mathrm{Lie} \left(  \{\mathcal{E}_{2k-1}(B^{(n)})\mid k=1,\ldots,n-1\}  \cup \{M+\mathcal{E}_{2n-1}(B^{(n)})\} \right)
 = \mathfrak{su}(n).
\end{align*}
Therefore, thanks to Theorems~\ref{metalemma} and~\ref{metalemma-tracking}, system~\eqref{EQ_potential_well} is approximately simultaneously controllable and a stalker.

\section{The 3D molecule}
\label{rotational}

Let us go back to the system presented in the introduction for the orientation of a linear molecule, that is,  
\begin{equation}\label{mol!}
i \hbar \dot \psi = -\Delta \psi + (u_1 \cos \theta + u_2 \cos \varphi\sin \theta+ u_3 \sin \varphi\sin \theta)\psi,
\end{equation}
where $\psi(t) \in \H=L^2({S}^2,\CC)$.

A basis of eigenvectors of the Laplace--Beltrami operator $\Delta$ is given by  
the spherical harmonics $Y^m_\l(\theta, \varphi)$, which satisfy
$$
\Delta Y^m_\l (\theta,\varphi) = -\l(\l+1)Y^m_\l(\theta, \varphi).
$$
The spectrum of $A = i\Delta$ is $\{ -i \l(\l+1)\mid \l\in \NN\}$. Each eigenvalue $-i \l(\l+1)$ is of finite multiplicity $2\l+1$.
Therefore $A$ satisfies Assumptions~$(\mathbb{A}1)$ and $(\mathbb{A}5)$.
Using the notations of the preceding  sections we set $B_{1}$, $B_{2}$, $B_{3}$ to be the multiplication operators by  $-i \cos \varphi\sin \theta$,  $-i\sin \varphi\sin \theta$, $-i \cos \theta$ respectively. 
Being $B_{1}$, $B_{2}$, $B_{3}$ bounded,  conditions $(\mathbb{A}2)$, $(\mathbb{A}3)$, and $(\mathbb{A}4)$ hold. Hence $(\mathbb{A}')$ is satisfied. Moreover, as proved in \cite[Proposition 8]{weakly}, \eqref{mol!} is $s$-weakly coupled for every $s>0$. 
The main goal of this section is to prove that system~\eqref{mol!} satisfies the \CSC. As a consequence we obtain the following result, whose corollary is Theorem~\ref{controllability}. 

\begin{theorem}
System \eqref{mol!} is:
\begin{enumerate} 
\item[$(i)$] \Lg\ 
approximately simultaneously controllable,
\item[$(ii)$]
approximately simultaneously controllable
 in $H^s$ for every $s> 0$, 
\item[$(iii)$] a stalker.
\end{enumerate}
\end{theorem}

Using classical identities for Legendre polynomials and trigonometric relations one can prove that
$$
\langle Y^{m}_{\l} , B_{j} Y^{m'}_{\l} \rangle = 0
$$
for every $j=1,2,3$, and $m,m' \in \{-\l-1, \ldots, \l+1\}$. 

Moreover
$$
\langle Y^{m}_{\l} , B_{j} Y^{m'}_{\l'} \rangle = 0
$$
with $|\l -\l'|\geq 2$ for every $m\in \{-\l-1, \ldots, \l+1\}$, $m'\in \{-\l'-1, \ldots, \l'+1\}$, $j=1,2,3$.
In order to prove that the \CSC\ is satisfied,  we  choose a reordering $(\phi_k)_{k\in \NN}$ of the spherical harmonics in such a way that 
$$\{\phi_k\mid k=1,\dots,4\l+4\}=\{Y^{-\l}_\l,\ldots,Y^\l_\l,Y^{-\l-1}_{\l+1},\ldots,Y^{\l+1}_{\l+1}\},$$
and we are left to prove that 
$$
\mathrm{Lie} \mathcal{V}_{4\l+4} = \mathfrak{su}(4\l+4).
$$
The characteristics spectral gap of the space $\mathcal{H}_\l$ is $(\l+1)(\l+2) -\l(\l+1)=2 (\l+1)$. In particular
$
(2 (\l+1),1) , (2 (\l+1),2)$, and $(2 (\l+1),3)$ are in  $\Xi_{4\l+4}$.

\subsection{Matrix representations}

Denote
by $\mathcal{J}_\l$ 
the set of integer pairs 
$\{(r,m)\mid r=\l,\l+1,\ m=-j,\dots,j\}$. 
Consider the lexicographic ordering $\varrho:\{1,\dots,4\l+4\}\to \mathcal{J}_\l$. 
For $j,k=1,\dots, 4\l+4$, let $e_{j,k}$ be the $(4\l+4)$-square matrix whose entries are all zero, but the one at line $j$ and column $k$ 
which is equal to $1$.
Define
$$
E_{j,k} = e_{j,k} - e_{k,j}, \  F_{j,k} = i e_{j,k} + i e_{k,j}, \ D_{j,k} = ie_{j,j} - i e_{k,k} .
$$

By a slight abuse of language, also set 
$e_{\varrho(j),\varrho(k)}=e_{j,k}$. The analogous identification can be used to define $E_{\varrho(j),\varrho(k)},F_{\varrho(j),\varrho(k)},D_{\varrho(j),\varrho(k)}$.
Note that
\begin{equation}\label{Ji}
J_{i}(E_{(\l,m),(\l+1,n)}) =- F_{(\l,m),(\l+1,n)},
\quad
\mbox{ 
and
} \quad
J_{i}(F_{(\l,m),(\l+1,n)}) = E_{(\l,m),(\l+1,n)}.
\end{equation}
Thanks to this notation we can conveniently represent the matrices corresponding 
to the 
controlled vector field (projected on $\mathcal{H}_\l$).
A computation shows that the 
control potentials in the $x$ and $y$ directions, $-i \cos \varphi\sin \theta$ and 
$-i \sin \varphi\sin \theta$ respectively, projected on $\mathcal{H}_\l$, have the matrix representations
\begin{align*}
B_{1}^{(4\l+4)} &= \sum_{m=-\l}^{\l}( - q_{\l,m} F_{(\l,m),(\l+1,m-1)} + q_{\l,-m} F_{(\l,m),(\l+1,m+1)})\\
B_{2}^{(4\l+4)} &= \sum_{m=-\l}^{\l}(  q_{\l,m} E_{(\l,m),(\l+1,m-1)} +q_{\l,-m} E_{(\l,m),(\l+1,m+1)}),
\end{align*}
where
$$
q_{\l,m}= \sqrt{\frac{(\l-m+2)(\l-m+1)}{4(2\l+1)(2\l+3)}}.
$$

Similarly, we associate with the 
control potential in the $z$ direction, $-i \cos \theta$ the matrix representation
$$
B_{3}^{(4\l+4)}= \sum_{m=-\l}^{\l} p_{\l,m} F_{(\l,m),(\l+1,m)},
$$
with
$$p_{\l,m}=- \sqrt{\frac{(\l+1)^{2} - m^{2}}{(2\l+1)(2\l+3)}}.$$

\subsection{Useful bracket relations }\label{useful}

From the identity
$$
[e_{j,k},e_{n,m}] = \delta_{kn}e_{j,m} - \delta_{jm}e_{n,k}
$$
we get the relations
 \begin{equation}\label{eq:::}
[E_{j,k}, E_{k,n}] = E_{j,n},\qquad [F_{j,k}, F_{k,n}] = -E_{j,n},\qquad [E_{j,k}, F_{k,n}] = F_{j,n},   
   \end{equation}
and
 \begin{equation}\label{eq:1231}
   [E_{j,k}, F_{j,k}] = 2D_{j,k}, \quad \mbox{ and  } \quad  [F_{j,k}, D_{j,k}] = 2E_{j,k}.
   \end{equation}
The relations above can be interpreted following 
a ``triangle rule'': 
the bracket between an operator coupling the states $j$ and $k$ and an operator coupling the states $k$ and $n$ couples the states $j$ and $n$. 
On the other hand, the bracket is zero if two operators couple no common states, that is,  
\begin{equation}\label{E000}
[Y_{j,k}, Z_{j',k'}]=0 \quad \mbox{ if } \{j,k\} \cap \{j',k'\} = \emptyset,
\end{equation}
with $Y,Z \in \{E,F,D\}$.

\subsection{Controllability in $\mathfrak{su}(4\l+4)$}

\begin{lemma}\label{prop:liemolecule}
The Lie algebra $L$ generated by $B_{1}^{(4\l+4)}, B_{2}^{(4\l+4)}, B_{3}^{(4\l+4)}$, $J_{i}(B_{1}^{(4\l+4)})$, $ J_{i}(B_{2}^{(4\l+4)}),$ $J_{i}(B_{3}^{(4\l+4)})$  is equal to $\mathfrak{su}(4\l+4)$.
\end{lemma}
\begin{proof}
The first step of the proof consists in showing that the Lie algebra $L$ contains the elementary matrices
\begin{equation}\label{Evicini}
E_{(\l,k),(\l+1,k+j)}\quad \mbox{ for }k=-\l,\dots,\l,\ j=-1,0,1.
\end{equation}

With a slight abuse of notation and for the sake of readability, let us write 
$B_{j} = B_{j}^{(4\l+4)}$, $j=1,2,3$. 
Let us also write $\ad_\alpha \beta$ for $[\alpha,\beta]$ and $\ad_\alpha^{j+1} \beta$ for $[\alpha,\ad_\alpha^{j} \beta]$.

Notice that
\begin{align*}
J_{i}(B_{3})  = \sum_{\l=-m}^{m} p_{\l,m} J_{i}(F_{(\l,m),(\l+1,m)}) =  \sum_{\l=-m}^{m} p_{\l,m} E_{(\l,m),(\l+1,m)}.
\end{align*}
By induction on $j\geq 0$ and using the bracket relations~\eqref{eq:1231}, 
we have
\begin{align*}
\ad^{2j}_{B_{3}} J_{i}(B_{3}) &= [B_{3},[B_{3},\ad^{2j-2}_{B_{3}} J_{i}(B_{3}) ] ] \\
 &= (-1)^{j}  2^{2j} \sum_{m=-\l}^{\l} p_{\l,m}^{2j+1} E_{(\l,m),(\l+1,m)}.
\end{align*}

By invertibility of the Vandermonde matrix and since
$p_{\l,m} \neq p_{\l,n}$ for every $n \neq m,-m$,
it follows
\begin{equation}\label{eq:verticali}
E_{(\l,-m),(\l+1,-m)}+E_{(\l,m),(\l+1,m)}\in L\qquad\mbox{ for } m=0,\dots,\l. 
\end{equation}

In particular $E_{(\l,0),(\l+1,0)}\in L$. The double bracket of  
 \begin{equation}\label{unaaa}
 \frac{B_{2} - J_{i}(B_{1})}{2} = \sum_{m=-\l}^{\l} q_{\l, m} E_{(\l,m),(\l+1,m-1)}\in L,
 \end{equation}
 with $E_{(\l,0),(\l+1,0)}$
 is easily computed using \eqref{eq:::} and \eqref{E000}
and gives 
\begin{align*}
\Big[\Big[ \sum_{m=-\l}^{\l} q_{\l, m} E_{(\l,m),(\l+1,m-1)}, E_{(\l,0),(\l+1,0)}\Big],E_{(\l,0),(\l+1,0)}\Big]&= 
 -q_{\l,1}[E_{(\l,0),(\l,1)},E_{(\l,0),(\l+1,0)}] \\
 &\quad  -q_{\l,0} [E_{(\l+1,-1),(\l+1,0)},E_{(\l,0),(\l+1,0)}]\\
&=  q_{\l,0}E_{(\l,0),(\l+1,-1)}+q_{\l,1} E_{(\l,1),(\l+1,0)} \in L.
\end{align*}

Define $Q_0= q_{\l,0}E_{(\l,0),(\l+1,-1)}+q_{\l,1} E_{(\l,1),(\l+1,0)}$ and, similarly,
$Q_m= q_{\l,-m}E_{(\l,-m),(\l+1,-m-1)} + q_{\l,m+1} E_{(\l,m+1),(\l+1,m)}$ for $0<m<\l$, $Q_\l= q_{\l,-\l}E_{(\l,-\l),(\l+1,-\l-1)}$.
In particular $ B_{2} - J_{i}(B_{1}) = 2\sum_{m=0}^{\l} Q_m$.

Using again \eqref{eq:::} and \eqref{E000}, we have
\begin{align*}
\Big[\Big[ \sum_{m=k}^\l Q_m, E_{(\l,-k),(\l+1,-k)}+ E_{(\l,k),(\l+1,k)} \Big], E_{(\l,-k),(\l+1,-k)}+E_{(\l,k),(\l+1,k)} \Big] 
= Q_{k},
\end{align*}
for $k=1,\dots,\l$. 
By recurrence on $k$ and because of \eqref{eq:verticali}, it follows that $Q_k\in L$ for $k=0,\dots,\l$.

Now, since  
$Q_\l/q_{\l,-\l}=E_{(\l,-\l),(\l+1,-\l-1)}$ is in $L$, then
\begin{align*}
\mathrm{ad}_{E_{(\l,-\l),(\l+1,-\l-1)}}^2 (E_{(\l,-\l),(\l+1,-\l)}+E_{(\l,\l),(\l+1,\l)})&=
  -E_{(\l,-\l),(\l+1,-\l)}\in L,
\end{align*}
which, in turns,  implies that 
\begin{align*}
\mathrm{ad}_{E_{(\l,-\l),(\l+1,-\l)}}^2 (Q_{\l-1})&=
 -q_{\l,-\l}E_{(\l,-\l+1),(\l+1,-\l)} \in L.
\end{align*}
Iterating the argument,  
 $E_{(\l,m),(\l+1,m)}$ and $E_{(\l,m),(\l+1,m-1)}$ are in $L$ for every $m= -\l,\ldots,\l$.

Developing the same argument as above replacing \eqref{unaaa} by 
$$
 \frac{B_{2} + J_{i}(B_{1})}{2} = \sum_{m=-\l}^{\l} q_{\l,-m} E_{(\l,m),(\l+1,m+1)}\in L,
 $$
we have that also  
$E_{(\l,m),(\l+1,m+1)}$ is in $L$ for every $m= -\l,\ldots,\l$,
proving \eqref{Evicini}. 
It then follows from \eqref{eq:::} that each $E_{j,k}$ is in $L$. 

If now we replace  \eqref{unaaa} by 
$$
\frac{B_{1}+J_{i}(B_{2})}{2} = -\sum_{m=-\l}^{\l} q_{\l,m} F_{(\l,m),(\l+1,m-1)}\in L,
$$
or
$$
\frac{B_{1}-J_{i}(B_{2})}{2} = \sum_{m=-\l}^{\l} q_{\l,-m} F_{(\l,m),(\l+1,m+1)}\in L,
$$
we obtain from the arguments above that $F_{(\l,m),(\l+1,m-1)}$ and $F_{(\l,m),(\l+1,m+1)}$ are in $L$ for every $m= -\l,\ldots,\l$.
The relations \eqref{eq:::} and \eqref{eq:1231} allow then to conclude that $L=\mathfrak{su}(4\l+4)$.
\end{proof}

\section{Proof of Theorem~\ref{metalemma}}\label{sec:prova1}

\subsection{Time-reparametrization}\label{s-time}

Up to replacing each $B_j$ by $\delta B_j$, we can assume that $\delta=1$.

For every piecewise constant function $z
$
such that
$z(t)\geq 1$ for every $t$,
we consider the time-reparametrization 
\begin{equation} \label{eq:main-repar0}
\frac{d\psi}{dt}(t)=(z(t)A +  u_1(t)z(t)B_1+ \cdots + u_p(t)z(t) B_p) \psi(t)
\end{equation}
of system~\eqref{eq:main}. Each $u_j(t)z(t)$ belongs to the time-varying set $z(t)U_j$. 

If $u_{1}, \ldots, u_{p}$ are control laws in~\eqref{eq:main-repar0} then the corresponding controls in~\eqref{eq:main} are their time-reparametrizations 
$\tilde u_{j}(s) = u_{j}(t(s))$  with $t(s)=\int_{0}^{s}z(\tau)d\tau$, $j=1,\ldots,p$.
By restricting the range of available controls and setting $v_{j}(t)=u_{j}(t)z(t)$,  we can focus our attention to trajectories of 
\begin{equation} \label{eq:main-repar}
\frac{d\psi}{dt}(t)=(z(t)A+v_1(t)B_1+ \cdots + v_p(t) B_p) \psi(t),\quad z(t)\geq 1,\quad v(t)=(v_1(t),\dots,v_p(t))\in U.
\end{equation}
Each solution of \eqref{eq:main-repar} with $z$ and $v$ piecewise constant is the time-reparametrization of a 
solution of~\eqref{eq:main} with piecewise constant controls (but the converse is not necessarily true, since we restricted the set of admissible controls). Hence, the approximate simultaneous controllability 
of \eqref{eq:main-repar}  implies the  
approximate  simultaneous controllability
of \eqref{eq:main}.
Moreover 
$$
\|\tilde u_{j}\|_{L^{1}}  = \int_{0}^{t^{-1}(T)}| \tilde u_{j}(\tau)|d\tau =  
\int_{0}^{T}|u_{j}(t)|z(t) dt
= \int_{0}^{T}{|v_{j}(t)|}dt \leq  T,
$$
for  $j=1,\ldots,p$. The last inequality holds since either $U_{j} = [0,1]$ or $U_{j}=[-1,1]$.
Hence 
the approximate simultaneous controllability in $\mathcal{L}$ of \eqref{eq:main-repar}  with a bound on the controllability time uniform with respect to the tolerance  
 implies, in fact, the \Lg\ approximate  simultaneous controllability
of \eqref{eq:main}.

\subsection{Interaction framework}
\label{sec:intfra}

Given a solution  $\psi(\cdot)$ of \eqref{eq:main-repar} with 
controls $z(\cdot), v_{1}(\cdot), \ldots, v_{p}(\cdot)$ and  
a piecewise constant function $\alpha(\cdot)$ with values in $\{0,1\}$, let us define
$$\omega(t) = \int_{0}^{t}(z(s)-\alpha(s))ds$$
and 
$$
y(t) = e^{-\omega(t)A } \psi(t).
$$
In particular
\begin{equation}\label{eq:1309}
|\langle \phi_k,y(t)\rangle | = |\langle \phi_k,\psi(t)\rangle|, \quad k\in \NN,
\end{equation}
for every $t$. 
For $
\omega, v_{1},\ldots,v_{p} \in \RR$ set $ \Theta(
\omega, v_{1},\ldots,v_{p})=
e^{-\omega A  }
(v_{1} {B_{1}} + \cdots + v_{p} {B_{p}} )
e^{\omega A }$.
Note that
\begin{align}
\Theta(\omega, v_{1},\ldots,v_{p})_{jk} &= \langle \phi_{k}, \Theta(\omega, v_{1},\ldots,v_{p}) \phi_{j} \rangle=
e^{i(\lambda_{k} - \lambda_{j})\omega} \left(v_{1}  (B_{1})_{jk}
+\cdots + v_{p} (B_{p})_{jk} \right),\label{m-Theta}
\end{align}
and that  $y(\cdot)$ satisfies
\begin{equation}\label{eq:main-interaction-pre}
\dot y(t)   = (\alpha(t) A+
\Theta(\omega(t), v_{1}(t),\ldots,v_{p}(t)))y(t), \qquad \alpha\in\{0,1\}, v\in U, \dot \omega + \alpha \geq 1.
\end{equation}
Conversely, each solution of \eqref{eq:main-interaction-pre} with $\alpha\in\{0,1\}$ and $v\in U$ piecewise constant 
and
$\omega$ continuous and piecewise affine, with $\dot \omega+\alpha =z\geq 1$ almost everywhere,
is, up to a 
time-dependent change of coordinates 
preserving the modulus of each component with respect to the 
basis $\Phi$,  a 
solution of~\eqref{eq:main-repar} with $u$ piecewise constant. 
In particular, each solution of  
\begin{equation}\label{eq:main-interaction}
\dot y(t)   = (\alpha(t) A+
\Theta(\omega(t), v_{1}(t),\ldots,v_{p}(t)))y(t),\qquad \alpha\in\{0,1\}, v\in U, \dot \omega \geq 1,
\end{equation}
with
$\alpha, v$ piecewise constant 
and
 $\omega$ continuous and piecewise affine
is, up to a 
time-dependent change of coordinates preserving the modulus of each component, a 
solution of~\eqref{eq:main-repar} with $u$ piecewise constant  (but the converse is not necessarily true). 

 \begin{proposition}\label{prop:implies}
Approximate simultaneous controllability of \eqref{eq:main-interaction} implies  approximate  simultaneous controllability of \eqref{eq:main}. If, moreover, approximate simultaneous controllability in $\mathcal{L} = \bigcup_{k \in \NN}  \mathrm{span}\{\phi_{1},\ldots,\phi_{k}\}$ of~\eqref{eq:main-interaction} is achieved with a uniform bound on time then~\eqref{eq:main} is \Lg\ approximate  simultaneous controllable.
 \end{proposition}

\begin{proof}
The strategy of the proof follows the idea of the proof of~\cite[Proposition~6.1]{ancoranoi}.

It follows from~\eqref{eq:1309} that  approximate simultaneous controllability of \eqref{eq:main-interaction}  implies  approximate  simultaneous controllability of \eqref{eq:main-repar} in modulus.

Moreover, because of the unitarity of the evolution, the approximate simultaneous controllability of \eqref{eq:main-interaction} is equivalent 
to the approximate simultaneous controllability of the system
$$
\dot y(t)   = - (\alpha(t) A+ \Theta(\omega(t), v_{1}(t),\ldots,v_{p}(t)))y(t),\qquad \alpha\in\{0,1\},\quad v\in U,
$$
which implies approximate simultaneous controllability  in modulus of the time-reversed of~\eqref{eq:main-repar}.

Take $r$ orthonormal initial conditions $\psi_0^1,\ldots,\psi_0^r$ and $r$ orthonormal final conditions  $\psi_1^1,\ldots,\psi_1^r$.  
Since the the spectrum of $A$ is infinite by Assumption~($\mathbb{A}1$) we can 
apply~\cite[Lemma~6.3]{ancoranoi} so that for every
 tolerance $\eta>0$ there exist $ k_1,\dots, k_r$ such that
$$
{\cal C}=\{e^{t A}\phi_{k_1}+ \cdots +e^{t A}\phi_{k_r}\mid t\in\RR\}.
$$
is $\eta$-dense in the torus
$$
{\cal T}=\{e^{\theta_1 A}\phi_{k_1}+ \cdots +e^{\theta_r A}\phi_{k_r}\mid \theta_1,\ldots, \theta_{r} \in \RR\}.
$$

By approximate simultaneous controllability  in modulus of \eqref{eq:main-repar}  it follows
that there exists an admissible control  $(z,v)$
steering simultaneously each $\psi_0^j$, for  $j=1,\ldots,r$,    $\eta$-close to 
 $e^{\theta_jA}\phi_{k_j}$ for some $\theta_1,\dots,\theta_r\in\RR$.
 
Similarly, by approximate simultaneous controllability  in modulus of the time-reversed of~\eqref{eq:main-repar} there exists an admissible control $(\tilde z, \tilde v)$   
steering  system~\eqref{eq:main-repar} simultaneously, for some $\tilde \theta_1,\dots,\tilde \theta_r\in\RR$,  
from $e^{\tilde \theta_1A}\phi_{ k_1}, \ldots,e^{\tilde \theta_rA}\phi_{ k_r}$
to an $\eta$-neighborhood of $\psi_1^j, \ldots, \psi_{1}^{r}$. 

Finally the concatenation of the control $(z,v)$, a control constantly equal to $(1,0)$ on a time interval of suitable length, and $(\tilde z, \tilde v)$ steers 
system~\eqref{eq:main-repar} simultaneously 
from   $\psi_0^j, \ldots, \psi_{0}^{r}$
to a $3 \eta$-neighborhood of $\psi_1^j, \ldots, \psi_{1}^{r}$. 
  
According to the conclusion of Section~\ref{s-time}, the approximate simultaneous controllability of~\eqref{eq:main-repar} implies approximate simultaneous controllability of~\eqref{eq:main}.
\end{proof}

\subsection{Galerkin approximation}

\begin{definition}
Let $N \in \NN$.  The \emph{Galerkin approximation}  of \eqref{eq:main-interaction}
of order $N$ is the system 
\begin{equation}\label{eq:galerkin-N}
\dot x =(\alpha A^{(N)}+ \Theta^{(N)}(\omega, v_{1},\ldots,v_{p})) x, \quad x \in \H,
\end{equation}
where $\Theta^{(N)}(\omega, v_{1},\ldots,v_{p})=\pi_N \Theta(\omega, v_{1},\ldots,v_{p}) \pi_N$. 
The controls $v$ are piecewise constant with values in $U$, while 
$\omega$ is continuous and piecewise affine, with $\dot \omega \geq 1$ almost everywhere.
\end{definition}

In the following section we recall a convexification result whose role is to identify the matrices that can be obtained by convexification of matrices of the form $\Theta^{(N)}(\omega, v_{1},\ldots,v_{p})$. Recall that the elements of $\Theta^{(N)}(\omega, v_{1},\ldots,v_{p})$ are described by \eqref{m-Theta}.

\subsection{Convexification}

The following technical result 
has been proved in \cite{ancoranoi}. 

\begin{lemma}\label{lem:convexification}
Let $\kappa$ be a positive integer and $\gamma_{1}, \ldots, \gamma_{\kappa} \in \RR\setminus \{0\}$ be  such that
$|\gamma_{1}|\neq|\gamma_{j}|$ for $j =2,\ldots, \kappa.$
Let 
$$ 
\varphi(t) = (e^{it\gamma_{1}}, \ldots, e^{it\gamma_{\kappa}}).
$$
Then, for every $\tau_0\in \RR$, we have
$$ 
\overline{\mathrm{conv}{\varphi([\tau_0,\infty))}} \supseteq \nnu S^{1} \times \{(0, \ldots, 0)\}\,,
$$
where
\begin{equation}\label{nu}
\nnu=\prod_{k=2}^{\infty} \cos \left (\frac{\pi}{2 k} \right ) >0.
\end{equation}
Moreover,
for every $R>0$ and $\xi \in {S}^{1}$ there exists a sequence $(\tau_{k})_{k=1}^\infty$ such that $\tau_1\geq \tau_0$, 
$
\tau_{k+1} - \tau_{k} > R
$,
and 
$$
\lim_{K\to \infty } \frac{1}{K} \sum_{k=1}^{K} \varphi (\tau_{k}) = (\nnu \xi,0,\ldots,0)\,.
$$
\end{lemma}


\subsection{Choice of the order of the Galerkin approximation
}\label{sec:n}

In order to prove approximate simultaneous controllability,  
we should take $r$ in $ \NN$, $\psi_1,\ldots,\psi_{r}$ in $\H$, $\hat\Upsilon$ in $\mathbf{U}(\H)$, and $\vep>0$ and prove 
 the existence of a piecewise constant control $u:[0,T]\rightarrow U$
 such that
$$
\left \| \hat\Upsilon \psi_k - \pro^{u}_T \psi_k \right \|<\vep,\qquad k=1,\dots,r.
$$

Notice that for  $n_{0}$ large enough there 
exists  $g
\in SU(n_{0})$ such that 
\begin{equation}\label{eq:1402}
| \langle \phi_{j},\hat\Upsilon \psi_k \rangle  -  \langle \pi_{n_{0}}\phi_{j},g \pi_{n_{0}}\psi_k \rangle| <\vep
\end{equation}
for every  $1 \leq k\leq r$ and $j\in\NN$.
This simple fact 
suggests to prove 
approximate simultaneous controllability
by studying the controllability of  the lift of
\eqref{eq:main-interaction} 
in the Lie group $SU(n_{0})$.

\subsection{Control in $SU(n)$}\label{Wn}

Let $n\geq n_0$ 
be chosen, in accord with the statement of Theorem~\ref{metalemma}, such that 
hypothesis \eqref{hypothesis} holds true.
Define the 
set of matrices 
\begin{align*}
 \mathcal{W}_{n}
 = &
\left\{A^{(n)} 
\right\}
\cup \left\{{\cal E}_0(B^{(n)}_j)\mid
(0,j)\in \Xi_n \right\}\\
&\cup \left\{{\cal E}_0(B^{(n)}_j)+\nnu{\cal E}_\sigma(B^{(n)}_j)\mid
(\sigma,j)\in \Xi_n 
\mbox{ and $\sigma,j$ are such that }
(0,j)\in\Xi_n, \sigma\ne 0
\right\}\\
&
\cup \left\{\nnu{\cal E}_\sigma(B^{(n)}_j)\mid
(\sigma,j)\in \Xi_n,\ \sigma\ne 0, 
\mbox{ and }
U_j=[-1,1]
\right\},
\end{align*}
where 
$\Xi_n$ and $\nnu$ are defined as in
\eqref{sigmabar} and \eqref{nu}, respectively. 
(Recall that by rescaling we are assuming $\delta=1$.)

Notice that $\mathrm{Lie}(\mathcal{W}_{n})=\mathrm{Lie}(\mathcal{V}_{n}^0)$.

Consider the 
auxiliary control system 
\begin{equation}\label{eq:Mausiliario}
\dot x= M(t) x,\quad M(t)\in \mathcal{W}_n,
\end{equation}
where $M$ plays the role of control. 
 It follows from \eqref{hypothesis} and standard controllability results on compact Lie groups (see \cite{jur}) that  
for every $g \in SU(n)$ 
there exists a piecewise constant function $M:[0,T] \to \mathcal{W}_n$ such that 
$$
 \rexp{T}{M(s)}{s} = g,
$$
where the chronological notation $\rexp{t}{V_s}{s}$ is used for the flow from time $0$ to $t$ of the time-varying equation $\dot q=V_s(q)$, $q \in \CC^{n}$ (see \cite{book2}).

\subsection{System reduction by convexification}\label{WnN}

Let $n$ be fixed as in the previous section. For every $N \geq n$ let
\begin{align*}
 \mathcal{W}_{n,N}
 = &
\left\{A^{(N)} 
\right\}
\cup \left\{{\cal E}_0(B^{(N)}_j)\mid
(0,j)\in \Xi_n \right\}\\
&\cup \left\{{\cal E}_0(B^{(N)}_j)+\nnu{\cal E}_\sigma(B^{(N)}_j)\mid
(\sigma,j)\in \Xi_n 
\mbox{ and $\sigma,j$ are such that }
(0,j)\in\Xi_n, \sigma\ne 0
\right\}\\
&
\cup \left\{\nnu{\cal E}_\sigma(B^{(N)}_j)\mid
(\sigma,j)\in \Xi_n,\ \sigma\ne 0, 
\mbox{ and }
U_j=[-1,1]
\right\}.
\end{align*}


\begin{lemma}\label{lem:0000}
For every $N \ge n$ and for every piecewise constant $M: [0,T] \to \mathcal{W}_{n,N}$  there exist 
$\alpha:[0,T]\to \{0,1\}$, $v:[0,T]\to U$ piecewise constant and a sequence $(\omega_{h}(\cdot))_{h\in \NN}$ of
 continuous and piecewise affine functions from $[0,T]$ to $[0, \infty)$ with $\dot \omega_h\ge 1$
 almost everywhere, such that
\begin{align*}
\left\|\int_{0}^{t} (\alpha(s)A^{(N)}+\Theta^{(N)}(\omega_{h}(s),v_{1}(s),\ldots,v_{p}(s))) ds \right.
 \left. - \int_{0}^{t}M(s)ds\right\| \to 0
\end{align*}
uniformly with respect to $t \in [0,T]$ as $h$ tends to infinity. 
\end{lemma}

\begin{proof}
Let $N \geq n$.
Let us fix $\alpha$ and $v_1,\dots,v_p$ at each $t\in[0,T]$ as follows: if $M(t)=A^{(N)}$ then $\alpha(t)=1$ and $v_1(t)=\dots=v_p(t)=0$; otherwise, if $M(t)=\mathcal{E}_{0}(B_j^{(N)})$, 
$M(t)=\mathcal{E}_0(B_j^{(N)})+\nnu \mathcal{E}_\sigma(B_j^{(N)})$, or $M(t)=\nnu\mathcal{E}_{\sigma}(B_j^{(N)})$
for some $j$, 
then take such a $j$ minimal and set $v_j(t)=1$ and $\alpha(t)=v_k(t)=0$ for $k\ne j$.

We are going to 
apply Lemma~\ref{lem:convexification} for
every interval on which $M(\cdot)$ is constant. 
Fix $\omega_h(0)=0$ for every $h$. Take an interval $(t_0,t_1)$ on which $M(\cdot)$ is constant and assume that $\omega_h(t_0)$ has been computed. We next extend $\omega_h$ on $(t_0,t_1)$. 

If $\alpha=1$ on $(t_0,t_1)$ then take $\omega_h(\tau)=\omega_h(t_0)+\tau-t_{0}$ for every $\tau \in (t_{0},t_{1})$.

Otherwise, let
$v_j=1$ on $(t_0,t_1)$ and assume for now that 
$M(t)=\mathcal{E}_0(B_j^{(N)})+\nnu \mathcal{E}_\sigma(B_j^{(N)})$. 
Apply Lemma~\ref{lem:convexification} with $\gamma_{1} = \sigma$, 
$\{\gamma_{2}, \ldots,\gamma_{\kappa}\} =  
\Sigma_{N} \setminus \{\sigma\}$, $\xi=1$,
$R=T$,
and $\tau_0 = \omega_h(t_0)$.
Then there exists a sequence 
$(\tau_{k})_{k=1}^{\infty}$
such that 
$\tau_1\geq \omega_h(t_0)$, 
$
\tau_{k+1} - \tau_{k} > T
$,
and 
$$
\lim_{K\to \infty } \frac{1}{K} \sum_{k=1}^{K} (e^{i\tau_k \gamma_{1}}, \ldots, e^{i\tau_k \gamma_{\kappa}}) = (\nnu ,0,\ldots,0).
$$
In particular there exists $K=K(h)$ such that
\begin{align*}
\left|\frac{1}{K} \sum_{k=1}^{K} e^{i(\lambda_{l} - \lambda_{m}) \tau_{k} } 
b_{ml}^{(j)} 
- \left( 
\mathcal{E}_0(B_j^{(N)})+\nnu \mathcal{E}_\sigma(B_j^{(N)}) 
\right)_{m,l}
\right| < \frac 1 h,
\end{align*}
for every $1 \leq l,m \leq N$.

Consider the piecewise constant function $Y: (t_{0},t_{1}) \to \RR$ defined as follows:
set $s_{\alpha} = t_{1}+ (t_{1} - t_{0}){\alpha/K}$, $\alpha = 0,\ldots, K$, and let
\begin{equation*}
Y(t) =\omega_{h}(t_{0}) + \sum_{\alpha=1}^{K} 
\tau_{\alpha}  \chi_{[s_{\alpha-1}, s_{\alpha})}(t)\,.
\end{equation*}

Following the smoothing procedure of~\cite[Proposition~5.5]{ancoranoi} one can construct  a continuous piecewise affine approximation $\omega_{h}:[t_{0},t_{1}] \to \RR$ of $Y$ with $\dot \omega_h\geq 1$ 
almost everywhere such that 
\begin{align}\label{c-i}
\left\|\int_{t_{0}}^{t} (\Theta^{(N)}(\omega_{h}(s),v_{1}(s),\ldots,v_{p}(s))) ds  - \int_{t_0}^{t}M(s)ds\right\| \to 0
\end{align}
uniformly with respect to $t \in [t_{0},t_{1}]$ as $h$ tends to infinity.

The same argument can be carried out in the case  in which $M(t)=\mathcal{E}_0(B_j^{(N)})$  by applying Lemma~\ref{lem:convexification} with $\gamma_{1} $ in $(0,\infty)\setminus\Sigma_{N}$, 
$\{\gamma_{2}, \ldots,\gamma_{\kappa}\} =  
\Sigma_{N}$, $\xi=1$,
$R=T$,
and $\tau_0 = \omega_h(t_0)$.

The final case to be considered is when $M(t)=\nnu\mathcal{E}_\sigma(B_j^{(N)})$ with $\sigma\ne 0$, $(\sigma,j)\in \Xi_n$, and 
$U_j=[-1,1]$. 
Notice that 
\begin{equation}\label{con-s}
\nnu\mathcal{E}_\sigma(B_j^{(N)})=\frac{(\mathcal{E}_0(B_j^{(N)})+\nnu \mathcal{E}_\sigma(B_j^{(N)}))-(\mathcal{E}_0(B_j^{(N)})+\nnu J_{-1}(\mathcal{E}_\sigma(B_j^{(N)})))}2.
\end{equation}

The argument above can be easily adapted to matrices $M(t)$ of the type  
$v_j(\mathcal{E}_0(B_j^{(N)})+\nnu J_{\xi}(\mathcal{E}_\sigma(B_j^{(N)})))$, with $v_j\in U_j$, $\xi\in S^1$ (just not imposing  $\xi=1$ while applying Lemma~\ref{lem:convexification}), and in particular to 
$-(\mathcal{E}_0(B_j^{(N)})+\nnu J_{-1}(\mathcal{E}_\sigma(B_j^{(N)})))$. 

It suffices then to introduce a sequence $(M^h)_{h\in \NN}$ of 
piecewise constant functions with values in 
$$
\{ v_j(\mathcal{E}_0(B_j^{(N)})+\nnu J_{\xi}(\mathcal{E}_\sigma(B_j^{(N)})))\mid v_j\in U_j,\ \xi\in S^1\}
$$
such that $\int_{t_0}^t M^h(s)ds $ converges uniformly for $t\in [t_0,t_1]$ to 
$\int_{t_0}^t M(s)ds $ as $h$ tends to infinity and to apply a diagonal procedure based on the approximation  introduced above. 
\end{proof}

As a consequence of the lemma above and thanks to~\cite[Lemma~8.2]{book2}, we have 
\begin{align}\label{justabove}
 \left\| 
 \rexp{t}{
 \left(\alpha(s)A^{(N)}+\Theta^{(N)}(\omega_{h}(s),v_{1}(s),\ldots,v_{p}(s))\right)}{s} 
- \rexp{t}{M(s)}{s}
 \right\|
 \to 0
\end{align}
uniformly with respect to $t \in [0,T]$  as $h$ tends to infinity.

\subsection{Control of the infinite-dimensional system}\label{sec:infinitedimension}
Next proposition states that we can pass to the limit as $N$ tends to infinity without losing the controllability property 
\eqref{justabove}.
Its proof is based on the special sparsity structure
of the matrices in  $\mathcal{W}_{n,N}$,
guaranteeing that the difference between the dynamics of the infinite-dimensional system and the dynamics of the Galerkin approximations is small.

We introduce the following notation: given $n \in \NN$ and a bounded linear transformation $L$ of $\H$, let $\crop_n(L)$ be the $n\times n$ matrix $(\langle \phi_j,L\phi_k\rangle)_{j,k=1}^n$. We use the same symbol $\crop_n$ to denote the similar cropping operation acting on the space of $N\times N$ matrices, with $N\geq n$.

\begin{proposition}\label{prop:2410}
Let $n \in \NN$ and $M:[0,T] \to \mathcal{W}_{n}$ be   piecewise constant. Then, for every $\varepsilon>0$, there exist piecewise constant controls $z:[0,T] \to [1,\infty)$ and $v:[0,T] \to U$
and a continuous piecewise affine function $\omega$ with $\dot \omega \geq 1$ 
such that
the propagator $\Psi$ of~\eqref{eq:main-interaction} satisfies
$$
\left\| 
\rexp{t}{M(s)}{s}   -  
\crop_{n} (\Psi_{t}) 
\right\|
<\varepsilon,
$$
for every 
$t \in [0,T]$.
\end{proposition}

\begin{proof} Consider $\mu > 0$ to be fixed later.
For every $j\in\NN$ the hypothesis that $\phi_j$ belongs to $D(B_{l})$ implies that
the sequence $((B_{l})_{jk})_{k\in \NN}$ is in $\ell^2$ for every $l=1,\ldots,p$. 
It is therefore possible to choose
$N\geq n $ such that
$\|((B_{l})_{jk})_{k>N}\|_{\ell^{2}} < \mu$
for every $j=1,\dots, n$ and $l=1,\ldots,p$. 

Let $\hat M$ be a piecewise constant function from $[0,T]$ to $\mathcal{W}_{n,N}$ such that $\crop_{n} \hat M(t) = M(t)$ for every $t$ in $[0,T]$.
Because of the definition of $\Xi_n$ and of 
the classes $\mathcal{W}_{n,N}$ and $\mathcal{W}_{n}$ we have

$$
\rexpi{s}{t}{\hat M(\tau)}{\tau}  = 
\left(  \begin{array}{c|c}
\displaystyle\rexpi{s}{t}{M(\tau)}{\tau} &\phantom{\int} 0\phantom{\int}\\[3mm]
\hline\\[-3mm]
0 & 
*
\end{array}
\right).
$$

By Lemma~\ref{lem:0000}, for every $\eta > 0$, there exist piecewise constant controls
$\alpha:[0,T]\to \{0,1\}$, $v:[0,T]\to U$ and a continuous piecewise affine function $\omega$ with $\dot \omega \geq 1$ 
such that 
$$
\left\| 
 \rexpi{s}{t}{
 \left(\alpha(\tau)A^{(N)}+\Theta^{(N)}(\omega(\tau),v_{1}(\tau),\ldots,v_{p}(\tau))\right)}{\tau} 
- \rexpi{s}{t}{\hat M(\tau)}{\tau}
 \right\| < \eta
 $$
for every $s,t$ in $[0,T]$.

Consider the solution $\Psi$ of \eqref{eq:main-interaction} associated with $\alpha$, $\omega$ and $v$. 
Set, for $k \in \NN$,
$$
Q^{(k)}_{t} = \crop_{k} \Psi_{t}.
$$
Now
$$
\dot Q^{(N)}_{t} = \left( \alpha A^{(N)}+ \Theta^{(N)}(\omega, v_{1},\ldots,v_{p})  \right) Q^{(N)}_{t} + R^{(N)}_{t},
$$
and the choice of $N$ is such that
\begin{equation}\label{eq:7575}
|(R^{(N)}_{t})_{j,k}| \leq \mu,
\end{equation}
for every $j=1,\ldots,n$ and $k =1,\ldots, N$.
Notice, moreover, that the norm of $R^{(N)}_{t}$ can be uniformly bounded by a positive constant $C$ independent of $\eta$ (possibly depending on  $N$ and 
hence on $\mu$).

By the variation formula and since $Q^{(n)}_{t} = \crop_{n}(Q^{(N)}_{t} )$ we have 
\begin{align*}
Q^{(n)}_{t} &= \crop_{n} 
\left[\rexp{t}{
 \left( \alpha(\tau) A^{(N)}+ \Theta^{(N)}(\omega(\tau), v_{1}(\tau),\ldots,v_{p}(\tau))  \right) 
}{\tau} \right.
+\\
&\quad
\left.
\int_{0}^{t}
\left( \rexpi{s}{t}{
 \left( \alpha(\tau) A^{(N)}+ \Theta^{(N)}(\omega(\tau), v_{1}(\tau),\ldots,v_{p}(\tau))  \right) 
}{\tau}
 \right)
 R^{(N)}_{s}
 ds
 \right],
\end{align*}
so that
\begin{align*}
\Big\|\crop_{n} & \left( \Psi_{t} - 
\rexp{t}{ \left( \alpha(\tau) A^{(N)}+ \Theta^{(N)}(\omega(\tau), v_{1}(\tau),\ldots,v_{p}(\tau))  \right)}{\tau} \right)\Big\|
\\
&\quad \leq   t \eta C + \left\|\int_{0}^{t} \left(\rexpi{s}{t}{M(\tau)}{\tau} \right)
\crop_{n} \left( R^{(N)}_{s} \right) ds\right\|.
\end{align*}
The norm of the matrix product 
$$
\rexpi{s}{t}{M(\tau)}{\tau} \crop_{n} \left( R^{(N)}_{s} \right)
$$
is equal to 
$$
\|\crop_{n} R^{(N)}_{s}\|. 
$$
The \emph{max} norm of $\crop_{n} R^{(N)}_{s}$ is smaller than 
$\mu$ as it follows from~\eqref{eq:7575}. Hence
\begin{align*}
\Big\|\crop_{n} & \left( \Psi_{t} - 
\rexp{t}{ \left( \alpha(\tau) A^{(N)}+ \Theta^{(N)}(\omega(\tau), v_{1}(\tau),\ldots,v_{p}(\tau))  \right)}{\tau} \right)\Big\|
\\
& \quad \leq T(\eta C + \sqrt{n} \mu).
\end{align*}
The constant  $T(\eta C + \sqrt{n} \mu)$ can be made arbitrarily small by choosing $\mu$ small with respect to $n$ and $T$  and then $\eta$ small with respect to $C=C(\mu)$ and $T$.
\end{proof}

\begin{proof}[Proof of Theorem~\ref{metalemma}]
Let $r$ in $ \NN$, $\psi_1,\ldots,\psi_{r}$ in $\H$, $\hat\Upsilon$ in $\mathbf{U}(\H)$, and $\vep>0$. 
Let $n$ be as in Section~\ref{Wn} and let $g \in SU(n)$ satify~\eqref{eq:1402}. 
Notice that if $\psi_1,\ldots,\psi_{r},\hat\Upsilon(\psi_1),\ldots,\hat\Upsilon(\psi_{r})$ are in $\mathcal{L}$ then $n$ can be taken independently of $\eps$. 

From Section~\ref{Wn}, there exists 
$M:[0,T] \to \mathcal{W}_{n}$ such that 
$$ 
\rexp{T}{M(s)}{s} =g.
$$

Proposition~\ref{prop:2410} ensures the existence of two piecewise constant functions $z$ and $v$ and of a continuous piecewise affine function $\omega$ with $\dot \omega \geq 1$ almost everywhere such that the associated propagator $\Psi$ of~\eqref{eq:main-interaction} satisfies
$$
\left\| 
\rexp{T}{M(s)}{s}   -  
\crop_{n} (\Psi_{T}) 
\right\|
<\varepsilon.
$$
If $\psi_1,\ldots,\psi_{r},\hat\Upsilon(\psi_1),\ldots,\hat\Upsilon(\psi_{r})$ are in $\mathcal{L}$ then $T$ is independent of $\eps$. 
By Proposition~\ref{prop:implies} system~\eqref{eq:main} is $L^1$-bounded approximately simultaneously controllable.
\end{proof}

\section{Proof of Theorem~\ref{metalemma-tracking}}\label{sec:tracking}

In order to prove Theorem~\ref{metalemma-tracking} we adapt 
the proof of Theorem~\ref{metalemma}.
The key point of the argument is the following: it has been proved in Proposition~\ref{prop:2410}  that system~\eqref{eq:main-repar}  can {track} every trajectory of~\eqref{eq:Mausiliario}. 
The idea is  
to replace \eqref{eq:Mausiliario} by a 
system which can {track} with arbitrary precision every trajectory in $SU(n)$. 
The crucial property,  beyond the Lie bracket generating condition,
that the new version of \eqref{eq:Mausiliario} should satisfy in order to 
achieve this goal is that it is a driftless system (i.e., the time-reversal of each of its  admissible trajectories is itself admissible).

The same time-reparameterization and time-dependent change of coordinates as in Section~\ref{s-time} allows to consider the tracking problem 
for system \eqref{eq:main-interaction}
instead of system~\eqref{eq:main}. As in the previous section we consider $\delta$ to be renormalized to $1$.

We can then base our argument on  the following analogue of Proposition~\ref{prop:implies}. 
 \begin{proposition}\label{prop:implies-track}
If \eqref{eq:main-interaction} is a stalker then \eqref{eq:main} is a stalker as well. 
 \end{proposition}

The following proposition allows to reduce a tracking problem 
in the space of unitary operators of $\H$ into 
a tracking problem in $SU(n)$ for $n$ large enough. Its proof can be found in \cite[Proposition~5.7]{ancoranoi}. 

\begin{proposition}\label{prop:tommaso}
Let $\hat\Upsilon:[0,T]\rightarrow \mathbf{U}(\H)$ be a continuous curve. 
Take $\vep>0$ and $m \in \NN$. 
Then for $n \geq m$ sufficiently large there exists 
a continuous curve $g:[0,T]\rightarrow SU(n)$ such that 
$| \langle \phi_{j},\hat\Upsilon(t) \phi_k \rangle  -  \langle e_{j},g(t) e_k \rangle| <\vep$ for every $t$ in $[0,T]$, $1 \leq k\leq m$, and $j=1,\dots,n$, where $e_1,\dots,e_n$ denotes the canonical basis of $\RR^n$.
\end{proposition}

Let $n$ 
be chosen as in Proposition~\ref{prop:tommaso}. In accord with the 
\CSC, we can assume, without loss of generality, that $\mathrm{Lie}(\mathcal{V}_n) = \mathfrak{su}(n)$.

The roles played in Sections~\ref{Wn} and \ref{WnN} by $\mathcal{W}_n$ and $\mathcal{W}_{n,N}$ are now played by 
$\nnu\mathcal{V}_n$ and 
$\nnu\mathcal{V}_{n,N}$ with 
\begin{align*}
 \mathcal{V}_{n,N}
 = &
 \left\{ J_{\xi}( {\cal E}_\sigma(B^{(N)}_j)) \mid
(\sigma,j)\in \Xi_n, \sigma\ne 0, \xi \in S^{1} \right\}.
\end{align*}

In particular, we consider as
auxiliary control system 
\begin{equation}\label{eq:Mausiliario-track}
\dot x= M(t) x,\qquad M(t)\in \nnu\mathcal{V}_n,
\end{equation}
 $M$ being the matrix-valued control parameter. 
 It follows from the equality 
 $\mathrm{Lie}(\mathcal{V}_n)=  \mathfrak{su}(n)$
and Rashevski--Chow's theorem  that  
 every trajectory on $SU(n)$ 
can be tracked with arbitrarily precision (up to time-reparameterization)
by a trajectory of \eqref{eq:Mausiliario-track}.

The relation between the trajectories of
\eqref{eq:galerkin-N} and those of \eqref{eq:Mausiliario-track} (or, more precisely, 
$\dot x= M(t) x$, $M(t)\in \nnu\mathcal{V}_{n,N}$), is described by the following lemma. 
\begin{lemma}\label{lem:0000-t}
For every $N \ge n$ and for every piecewise constant $M: [0,T] \to \nnu\mathcal{V}_{n,N}$  there exist 
$\alpha:[0,T]\to \{0,1\}$, $v:[0,T]\to U$ piecewise constant and a sequence $(\omega_{h}(\cdot))_{h\in \NN}$ of
 continuous and piecewise affine functions from $[0,T]$ to $[0, \infty)$ with $\dot \omega_h\ge 1$
 almost everywhere, such that
\begin{align*}
{\Big\|\int_{0}^{t} \left(\alpha(s)A^{(N)}+\Theta^{(N)}\left(\omega_{h}(s),
v_{1}(s),\ldots,v_{p}(s)\right)\right) ds}
-\int_{0}^{t}M(s)ds\Big\| \to 0
\end{align*}
uniformly with respect to $t \in [0,T]$ as $h$ tends to infinity. 
\end{lemma}

\begin{proof}
The proof is almost identical to that of Lemma~\ref{lem:0000} in the case $M(t)=\nnu\mathcal{E}_\sigma(B_j^{(N)})$. The only difference is in replacing \eqref{con-s} by
$$
\nnu J_\xi(\mathcal{E}_\sigma(B_j^{(N)}))=\frac{(\mathcal{E}_0(B_j^{(N)})+\nnu J_\xi(\mathcal{E}_\sigma(B_j^{(N)})))-(\mathcal{E}_0(B_j^{(N)})+\nnu J_{-\xi}(\mathcal{E}_\sigma(B_j^{(N)})))}2.
$$
We then apply the same convexification argument. 
\end{proof}

{
As in the previous section, 
the lemma above and~\cite[Lemma~8.2]{book2} imply that
\begin{align*}
\Big\| 
 \rexp{t}{
 \left(\alpha(s)A^{(N)}+\Theta^{(N)}(\omega_{h}(s),
 v_{1}(s),\ldots,v_{p}(s))\right)}{s}
 - \rexp{t}{M(s)}{s}
 \Big\|
 \to 0
\end{align*}
uniformly with respect to $t \in [0,T]$  as $h$ tends to infinity. 
}

In analogy with Section~\ref{sec:infinitedimension}
we can conclude the proof of Theorem~\ref{metalemma-tracking} thanks to the proposition below, which states that we can pass to the limit as $N$ tends to infinity without losing the tracking property of the finite-dimensional Galerkin approximations. Its proof is basically the same as that of Proposition~\ref{prop:2410}.

\begin{proposition}\label{prop:2410-t}
Let $n \in \NN$ and $M:[0,T] \to \nnu\mathcal{V}_{n}$ be   piecewise constant. Then, for every $\varepsilon>0$, there exist piecewise constant controls $z:[0,T] \to [1,\infty)$ and $v:[0,T] \to U$
and a continuous piecewise affine function $\omega$ with $\dot \omega \geq 1$ almost everywhere
such that
the propagator $\Psi$ of~\eqref{eq:main-interaction} satisfies
$$
\left\| 
\rexp{t}{M(s)}{s}   -  
\crop_{n} (\Psi_{t}) 
\right\|
<\varepsilon,
$$
for every 
$t \in [0,T]$.
\end{proposition}

\begin{remark}
The hypothesis that each $U_j$ contains $0$ in its interior can be relaxed.  
Indeed,  up to reordering, let $p'$ be such that $U_j=[0,\delta]$ if $j=1,\dots,p'$ and $U_j=[-\delta,\delta]$ for $j>p'$. 
Assume that,  for every $j\in\{1,\dots,p'\}$, if
 $l\neq k$ are such that $\lambda_{l} = \lambda_{k}$, then 
$\langle \phi_{l},B_{j}\phi_{k} \rangle = 0$. 
Assume, moreover,  that the \CCC\ is satisfied with $\mathcal{V}_n$ replaced by 
the set of all matrices 
$J_\xi(\mathcal{E}_\sigma(B^{(n)}_j))$ with $(\sigma,j)\in \Xi_n$, $\xi\in S^1$, $\sigma\ne 0$, 
and either $j>p'$ or
the following holds: 
if $l,k,l',k'\in\{1,\dots,n\}$ satisfy 
$\lambda_l-\lambda_k=\lambda_{l'}-\lambda_{k'}=\sigma$ and 
$$\langle \phi_{l},B_{j}\phi_{k} \rangle \ne 0 \ne \langle \phi_{l'},B_{j}\phi_{k'} \rangle $$
then 
$$\langle \phi_{l},B_{j}\phi_{l} \rangle-\langle \phi_{k},B_{j}\phi_{k} \rangle = \langle \phi_{l'},B_{j}\phi_{l'} \rangle-\langle \phi_{k'},B_{j}\phi_{k'} \rangle.$$

In this case the proof of 
Lemma~\ref{lem:0000-t} becomes more technically involved. 
The point is that, even if $\mathcal{E}_0(B_j^{(N)})$ cannot be eliminated by convexification, it is a diagonal matrix by hypothesis. Hence, it can be used to define a new interaction framework.
The sequence $(\omega_h(\cdot))_h$ can then be constructed by following 
\cite[Proposition 5.5]{ancoranoi}.
\end{remark}

\section{Proof of Theorem~\ref{thm:controlHs}}\label{6}

First, let us prove \Lg\ approximate simultaneous controllability in $s/2$-norm for initial and final data in $\mathcal{L}$. Namely we want to prove that, for $r \in \NN$, $\psi_1,\ldots,\psi_r \in \mathcal{L}$, and $\hat{\Upsilon} \in \mathbf{U}(\H)$ 
 with $\hat\Upsilon \psi_{1}, \ldots, \hat\Upsilon \psi_{r} \in \mathcal{L}$,
 there exists $K>0$ such that the following holds: For every $\varepsilon>0$ there exists a control $u$, with $\|u\|_{L^{1}} \leq K$ such that 
\begin{equation}\label{eq:2200}
\|\psi_{j} -  \Upsilon_{T}^{u}\psi_{j}\|_{s/2} < \varepsilon \qquad j = 1,\ldots,r.
\end{equation}

 Let $N$  be such that  $\psi_1,\ldots,\psi_r$ and  $\hat\Upsilon \psi_{1}, \ldots, \hat\Upsilon \psi_{r}$ are in $\mathrm{span}\{\phi_{1},\ldots,\phi_{N}\}$.
Note that on $\mathrm{span}\{\phi_{1},\ldots,\phi_{N}\}$ we have 
\begin{equation}\label{eq:equivalenza}
\|\psi\|_{s/2}  = \left(\sum_{k=1}^{N} |\lambda_{k}|^{s} |\langle \phi_{k},\psi \rangle|^{2}\right)^{1/2} 
\leq (\max\{|\lambda_{1}|,|\lambda_{N}|\})^{s/2} \|\psi\|.
\end{equation}
Since the system is \Lg\ approximately simultaneously controllable
there exists $K >0$ such that for every $\varepsilon>0$ there exists a piecewise constant control $u$, with $\|u\|_{L^{1}} \leq K$ such that
$$
\|\psi_{j} -  \Upsilon_{T}^{u}\psi_{j}\| < \varepsilon \qquad j = 1,\ldots,r.
$$
Hence, by~\eqref{eq:equivalenza}, we deduce  
\Lg\ approximate simultaneous controllability in $s/2$-norm in $\mathcal{L}$. 

Now let $\psi_1,\ldots,\psi_r \in D(|A|^{s/2})$, and $\hat{\Upsilon} \in \mathbf{U}(\H)$ be such that 
$\hat\Upsilon \psi_{1}, \ldots, \hat\Upsilon \psi_{r} \in D(|A|^{s/2})$.  Let $\varepsilon > 0$ and consider $\psi_{1}^{0},\ldots,\psi^{0}_r$  and $\psi_{1}^{1},\ldots,\psi^{1}_r$ in $\mathcal{L}$ such that
$$
\|\psi_{j}^{0} - \psi_j \|_{s/2} < \varepsilon \quad  \mbox{ and } \quad \|\psi_{j}^{1} - \hat\Upsilon\psi_j \|_{s/2}<\varepsilon, 
$$
for $j=1,\ldots,r$.
As proved above (see~\eqref{eq:2200}), there exist $K$, independent on $\varepsilon$, and a piecewise constant control $u$ with $\|u\|_{L^{1}} \leq K$ such that
$$
\|\psi_{j}^{1} -  \Upsilon_{T}^{u}\psi_{j}^{0}\|_{s/2} < \varepsilon \qquad j = 1,\ldots,r.
$$
By \cite[Proposition~2]{weakly}, since the system is $s$-weakly coupled, there exists a constant $C$ depending only on $s$ and $A,B_{1},\ldots,B_{p}$ such that
$$
\|\Upsilon_{T}^{u} \psi \|_{s/2} \leq C^{K}\|\psi\|_{s/2},
$$
for every $\psi \in D(|A|^{s/2})$.
Therefore
\begin{align*}
\|\Upsilon_{T}^{u}(\psi_{j}) - \hat\Upsilon(\psi_{j})\|_{s/2} & 
\leq \|\Upsilon_{T}^{u}(\psi_{j} - \psi^{0}_{j})\|_{s/2} + \| \Upsilon_{T}^{u}\psi_{j}^{0} - \psi_{j}^{1} \|_{s/2} + \|\psi_{j}^{1} - \hat\Upsilon\psi_j \|_{s/2}\\
& \leq (C^{K}+2) \varepsilon,
\end{align*}
for $j=1,\ldots,r$.

\begin{remark}
Using arguments similar to those of the proof of Theorem~\ref{metalemma} and of Theorem~\ref{thm:controlHs} it is possible to prove a finer statement 
 than Corollary~\ref{cor:qwerty}. Indeed it is possible to prove that a system satisfying Assumptions~($\mathbb{A}'$), the \CSC, and which is $s$-weakly coupled is a stalker for the $s/2$-norm.
This is due to the fact that, actually, the \CSC implies stalking in $\mathcal{L}$ with a uniform bound on the $L^{1}$ norm of the control. 
\end{remark}

\bibliographystyle{alpha}
\bibliography{biblio}

\end{document}